\documentclass[twoside,leqno,10pt, A4]{amsart}
\DeclareMathAlphabet{\mathpzc}{OT1}{pzc}{m}{it}
\usepackage{amsfonts}
\usepackage{amsmath}
\usepackage{amscd}
\usepackage{amssymb}
\usepackage{amsthm}
\usepackage{amsrefs}
\usepackage{latexsym}
\usepackage{mathrsfs}
\usepackage{bbm}
\usepackage{enumerate}
\usepackage[mathscr]{eucal}
\usepackage{color}
\begin{document}

\newtheorem{theorem}[subsection]{Theorem}
\newtheorem{proposition}[subsection]{Proposition}
\newtheorem{lemma}[subsection]{Lemma}
\newtheorem{corollary}[subsection]{Corollary}
\newtheorem{conjecture}[subsection]{Conjecture}
\newtheorem{prop}[subsection]{Proposition}
\numberwithin{equation}{section}
\newcommand{\mr}{\ensuremath{\mathbb R}}
\newcommand{\mc}{\ensuremath{\mathbb C}}
\newcommand{\dif}{\mathrm{d}}
\newcommand{\intz}{\mathbb{Z}}
\newcommand{\ratq}{\mathbb{Q}}
\newcommand{\natn}{\mathbb{N}}
\newcommand{\comc}{\mathbb{C}}
\newcommand{\rear}{\mathbb{R}}
\newcommand{\prip}{\mathbb{P}}
\newcommand{\uph}{\mathbb{H}}
\newcommand{\fief}{\mathbb{F}}
\newcommand{\majorarc}{\mathfrak{M}}
\newcommand{\minorarc}{\mathfrak{m}}
\newcommand{\sings}{\mathfrak{S}}
\newcommand{\fA}{\ensuremath{\mathfrak A}}
\newcommand{\mn}{\ensuremath{\mathbb N}}
\newcommand{\mq}{\ensuremath{\mathbb Q}}
\newcommand{\half}{\tfrac{1}{2}}
\newcommand{\f}{f\times \chi}
\newcommand{\summ}{\mathop{{\sum}^{\star}}}
\newcommand{\chiq}{\chi \bmod q}
\newcommand{\chidb}{\chi \bmod db}
\newcommand{\chid}{\chi \bmod d}
\newcommand{\sym}{\text{sym}^2}
\newcommand{\hhalf}{\tfrac{1}{2}}
\newcommand{\sumstar}{\sideset{}{^*}\sum}
\newcommand{\sumprime}{\sideset{}{'}\sum}
\newcommand{\sumprimeprime}{\sideset{}{''}\sum}
\newcommand{\sumflat}{\sideset{}{^{\flat}}\sum}
\newcommand{\sumSTAR}{\sideset{}{^{\star}}\sum}
\newcommand{\shortmod}{\ensuremath{\negthickspace \negthickspace \negthickspace \pmod}}
\newcommand{\V}{V\left(\frac{nm}{q^2}\right)}
\newcommand{\sumi}{\mathop{{\sum}^{\dagger}}}
\newcommand{\mz}{\ensuremath{\mathbb Z}}
\newcommand{\leg}[2]{\left(\frac{#1}{#2}\right)}
\newcommand{\muK}{\mu_{\omega}}

\newcommand{\RR}{\mathbb{R}}
\newcommand{\QQ}{\mathbb{Q}}
\newcommand{\CC}{\mathbb{C}}
\newcommand{\NN}{\mathbb{N}}
\newcommand{\ZZ}{\mathbb{Z}}
\newcommand{\FF}{\mathbb{F}}
\newcommand{\C}{{\mathcal{C}}}
\newcommand{\OO}{{\mathcal{O}}}
\newcommand{\cc}{{\mathfrak{c}}}
\newcommand{\norm}{{\mathpzc{N}}}
\newcommand{\trace}{{\mathrm{Tr}}}
\newcommand{\ringO}{{\mathfrak{O}}}
\newcommand{\fa}{{\mathfrak{a}}}
\newcommand{\fb}{{\mathfrak{b}}}
\newcommand{\fc}{{\mathfrak{c}}}
\newcommand{\res}{{\mathrm{res}}}
\newcommand{\fp}{{\mathfrak{p}}}
\newcommand{\fm}{{\mathfrak{m}}}
\newcommand{\aut}{\rm Aut}
\newcommand{\mt}{m(t,u;\ell^k)}
\newcommand{\mtone}{m(t_1,u;\ell^k)}
\newcommand{\mttwo}{m(t_2,u;\ell^k)}
\newcommand{\mbadu}{m(t,u_0;\ell^k)}
\newcommand{\Sts}{S(t_1,t_2;\ell^k)}
\newcommand{\Stt}{S(t;\ell^k)}
\newcommand{\St}{R(t;\ell^k)}
\newcommand{\nN}{n(N,u;\ell^k)}
\newcommand{\Tnn}{T(N;\ell^k)}
\newcommand{\Tn}{T(N;\ell^k)}
\newcommand{\tilnul}{{\tilde{\nu}}_\ell}
\newcommand{\tilnulk}{{\tilde{\nu}}_\ell^{(k)}}
\makeatletter
\def\imod#1{\allowbreak\mkern7mu({\operator@font mod}\,\,#1)}
\makeatother

\title[The Large Sieve with Power Moduli in Imaginary Quadratic Number Fields]{The Large Sieve with Power Moduli in Imaginary Quadratic Number Fields}

\date{\today}
\author{Peng Gao and Liangyi Zhao}

\begin{abstract}
   We establish large sieve inequalities for power moduli in imaginary quadratic number fields, extending earlier work of Baier and Bansal \cites{Baier&Bansal, Baier&Bansa2} for the Gaussian field.
\end{abstract}

\maketitle

\noindent {\bf Mathematics Subject Classification (2010)}: 11N35, 11L40 \newline

\noindent {\bf Keywords}: Large sieve, number fields, power moduli, prime moduli

\section{Introduction}
The classical large sieve inequality, a very useful tool with a wide range of applications in analytic number theory, originated from J. V. Linnik's study \cite{JVL1} on the distribution of quadratic non-residues.  There have been many subsequent refinements and extensions on the large sieve. One direction of investigating the large sieve is to establish such results for sparse sets of moduli. For prime moduli, this was obtained by D. Wolke in \cite{Wol}.  In recent years, the large sieve for moduli that are values of polynomials with degree at least two was studied in a series of papers \cites{Ba1, B&Z1, B&Z, Halupczok, Zha, ShpZha, BaiLynZha}).  \newline

    In \cite{Zha}, the second-named author conjectured that the following large sieve inequality holds for $k$-th power moduli ($k\in \mathbb{N}$ arbitrary but fixed):
\begin{equation*}
%% \label{conjec}
\sum\limits_{q\le Q} \sum\limits_{\substack{a=1\\ (a,q)=1}}^{q^k} \left| \sum\limits_{M<n\le M+N} a_n e\left( \frac{an}{q^k}\right) \right|^2
\ll_{\varepsilon} Q^{\varepsilon}\left(Q^{k+1}+N\right)\sum\limits_{M<n\le M+N} |a_n|^2.
\end{equation*}
   Here $Q, N \in \mn, M \in \mz$, $\varepsilon$ is any positive constant, and $\{a_n \}$ is any arbitrary sequence of complex numbers.  Additionally, it is proved in the same paper \cite{Zha} that
\begin{equation} \label{kls}
\begin{split}
& \sum\limits_{q\le Q} \sum\limits_{\substack{a=1\\ (a,q)=1}}^{q^k} \left| \sum\limits_{M<n\le M+N} a_n e\left(\frac{an}{q^k}\right)
\right|^2 \ll_{\varepsilon}  (QN)^{\varepsilon}\left(Q^{k+1}+NQ^{1-1/\kappa}+N^{1-1/\kappa}Q^{1+k/\kappa}
\right) \sum\limits_{M<n\le M+N} |a_n|^2,
\end{split}
\end{equation}
where $\kappa=2^{k-1}$. Improvements of this result have been established in \cites{B&Z1, B&Z, Halupczok}. \newline

An analogue result of \eqref{kls} was established for the Gaussian field by S. Baier and A. Bansal \cite{Baier&Bansal}, showing that
\begin{equation*}
\begin{split}
& \sum\limits_{\substack{q\in \mathbb{Z}[i]\setminus\{0\}\\ \norm(q)\le Q}}
\sum\limits_{\substack{r \bmod{q^k}\\ (r,q)=1}} \left|\sum\limits_{\substack{n\in \mathbb{Z}[i]\\ \norm(n)\le N}}
a_n \cdot e\left(\Re\left(\frac{nr}{q^k}\right)\right)\right|^2 \ll
(QN)^{\varepsilon}\left(Q^{k+1}+NQ^{1-1/\kappa}+N^{1-1/\kappa}Q^{1+k/\kappa}
\right)
\sum\limits_{\substack{n\in \mathbb{Z}[i]\\ \norm(n)\le N}} |a_n|^2,
\end{split}
\end{equation*}
  where $\Re(z)$ denotes the real part of $z$ for any $z \in \mc$ and $\norm(n)$ the norm the element $n$ in the number field.  A further improvement of the above result for square moduli was given recently in \cite{Baier&Bansa2}. \newline

  Motivated by the above results, we are interested in large sieve results for power moduli in imaginary quadratic number fields.
 Throughout the paper, we let $K$ be such a field and write $\mathcal{O}_K$ for the ring of integers in $K$.  It is well-known that we have $K=\mq(\sqrt{d})$ with $d$ a negative, square-free rational integer.  Then (see \cite[Section 3.8]{iwakow}) the discriminant $D_K$ of $K$ is
\begin{align*}
   D_K & =\begin{cases}
     d \qquad & \text{if $d \equiv 1 \pmod 4$ }, \\
     4d \qquad & \text{if $d \equiv 2, 3 \pmod 4$ }.
    \end{cases}
\end{align*}

Let $\norm(q)$ and ${\rm Tr}(q)$ denote the norm and the trace, respectively, of $q\in \mathcal{O}_K$. For any complex number $z$,  we define
\begin{align*}
   \widetilde{e}_K(z) =\exp \left( \mbox{\rm Tr} \left( \frac {z}{\sqrt{D_K}} \right) \right)=\exp \left( 2\pi i  \left( \frac {z}{\sqrt{D_K}} - \frac {\overline{z}}{\sqrt{D_K}} \right) \right).
\end{align*}

   To obtain a large sieve result for any number field, our first observation is that it is more proper to use the additive character $\widetilde{e}_K(z)$ instead of $e(\Re(z))$ in the general case, as $\widetilde{e}_K(z)$ naturally appears in the arithmetic of number fields. For example, in the definition of the Gauss sum associated to Hecke characters (see \cite[(3.86)]{iwakow}). Next, we note that a generalization of the large sieve for number fields was established by M. N. Huxley \cite{Huxley10}. In the case of imaginary quadratic number fields $K$, it takes the form
\begin{equation} \label{Huxley}
\sum\limits_{\substack{q\in \mathcal{O}_K\setminus\{0\}\\ \norm(q)\le Q}}
\sum\limits_{\substack{r \bmod{q}\\ (r,q)=1}} \left|\sum\limits_{\substack{n\in \mathcal{O}_K\\ \norm(n)\le N}}
a_n \cdot \widetilde{e}_K\left(\frac{nr}{q}\right)\right|^2 \ll \left(Q^2+N\right)\sum\limits_{\substack{n\in \mathcal{O}_K\\
\norm(n)\le N}} |a_n|^2.
\end{equation}

  In this paper, we first extend the above mentioned result of Baier and Bansal on large sieve for power moduli in the Gaussian field to all imaginary quadratic number fields. Our result is
\begin{theorem}
\label{mainthm}
   Let $K$ be any imaginary quadratic number field. Let $k\in \mathbb{N}, \kappa=2^{k-1}$, $Q,N\ge 1$ and $(a_n)_{n\in \mathcal{O}_K}$ be any sequence of complex numbers. Then
\begin{equation*}
\begin{split}
& \sum\limits_{\substack{q\in \mathcal{O}_K\setminus\{0\}\\ \norm(q)\le Q}}
\sum\limits_{\substack{r \bmod{q^k}\\ (r,q)=1}} \left|\sum\limits_{\substack{n\in \mathcal{O}_K\\ \norm(n)\le N}}
a_n \cdot \widetilde{e}_K\left(\frac{nr}{q^k}\right)\right|^2 \ll
(QN)^{\varepsilon}\left(Q^{k+1}+NQ^{1-1/\kappa}+N^{1-1/\kappa}Q^{1+k/\kappa}
\right)
\sum\limits_{\substack{n\in \mathcal{O}_K\\ \norm(n)\le N}} |a_n|^2,
\end{split}
\end{equation*}
where $\varepsilon$ is any positive constant, and the implied $\ll$-constant depends on $k$ and $\varepsilon$.
\end{theorem}

    Our proof of Theorem \ref{mainthm} follows along similar lines as in \cite{Baier&Bansal}. In particular, we apply Poisson summation over number fields to treat the related counting problem. The choice of the additive character $\widetilde{e}_K(z)$ allows us to present our arguments more concisely. \newline

    We note that in the case of the Gaussian field, an improvement of Theorem \ref{mainthm} for the case of square moduli was given by Baier and Bansal in \cite[Theorem 3]{Baier&Bansa2} recently in the following form:
\begin{align}
\label{sqrimprv}
\sum\limits_{\substack{q\in \mathbb{Z}[i]\setminus\{0\}\\ \norm(q)\le Q}}
\sum\limits_{\substack{a \bmod{q^2}\\ (a,q)=1}} \left|\sum\limits_{\substack{n\in\mathbb{Z}[i]\\ \norm(n)\le N}}a_n\cdot e\left(\Re\left(\frac{na}{q^2}\right)\right)\right|^2
 \ll (QN)^{\varepsilon}\left(Q^3 + Q^2\sqrt{N} + N\right)\sum\limits_{\substack{n\in \mathbb{Z}[i]\\ \norm(n)\le N}} |a_n|^2.
 \end{align}

   In fact, a number of results are obtained in \cite{Baier&Bansa2} and we show in what follows that these results can be at least extended to the case of all imaginary quadratic number fields of class number one. To do so, we need the following generalizations of the notations introduced in \cite{Baier&Bansa2}. \newline

Let $|z|$ denote the modulus of $z \in \mc$ as a complex number and
 $$
 B(y,u)=\{z\in \mathbb{C} : |z-y|\le u\}$$
   for the closed ball with center $y$ and radius $u$.
   We let $\mathcal{S}$ be any set satisfying 
 $$\mathcal{S} \subseteq B(0,Q^{1/2}) \cap (\mathcal{O}_K\setminus \{0\}). $$
For any $t\in \mathcal{O}_K\setminus\{0\}$, we define
 $$
 \mathcal{S}_t = \{q\in \mathcal{O}_K : tq \in \mathcal{S}\}.
 $$
   We further define
 \begin{equation*}
%%\label{Atdef}
 A_t(u,k,l) = \sup\limits_{\substack{y\in \mathbb{C}\\  |y| \le \frac{\sqrt{Q}}{|t|}}}\left|\{ q\in \mathcal{S}_t\cap B(y,u) : q\equiv l\bmod{k}\}\right|,
 \end{equation*}
 where $0\le u\le \sqrt{Q}/|t|$, $k\in \mathcal{O}_K\setminus\{0\}$ and $l\in \mathcal{O}_K$ with $(k,l) = 1$. \newline

    Our next two results in this paper extend the large sieve inequality given in \cite[Theorem 1]{Baier&Bansa2} for general sets $\mathcal{S}$ of moduli to any imaginary quadratic number field.
\begin{theorem}
\label{secdthm} Let $K$ be any imaginary quadratic number field. We have
\begin{equation*}
\begin{split}
\sum\limits_{q\in \mathcal{S}} &
 \sum\limits_{\substack{a \bmod{q}\\ (a,q)=1}} \left|\sum\limits_{\substack{n\in \mathcal{O}_K\\ \norm(n)\le N}}
a_n \cdot \widetilde{e}_K\left(\frac{na}{q}\right)\right|^2 \ll \\ &  N \left( 1 + \sup\limits_{\substack{r\in \mathcal{O}_K\setminus\{0\}\\ 1 \le |r| \le N^{1/4}}}
\sup\limits_{\substack{z\in \mathbb{C}\\ \frac{1}{N^{1/2}}\le |z| \le \frac{\sqrt{|D_K|}}{|r|N^{1/4}}}} \sup\limits_{\substack{h \in \mathcal{O}_K \\ (h,r) = 1}} \sum\limits_{t|r}
 \sum\limits_{\substack{m \in \mathcal{O}_K \\ 0<|m|\le \frac{3|rz|\sqrt{Q}}{|t|} \\ (m,\frac{r}{t}) = 1}}
 A_t \left( \frac{\sqrt{Q}}{\sqrt{N}|zt|},\frac{r}{t},hm \right)\right)\sum\limits_{\substack{n\in \mathcal{O}_K\\ \norm(n)\le N}} |a_n|^2.
\end{split}
\end{equation*}
\end{theorem}

   Under certain conditions on the size of $A_t(u,k,l)$, we derive from Theorem \ref{secdthm} the following
\begin{theorem}
\label{thirdthm}
 Let $K$ be any imaginary quadratic number field. Suppose that for all $t$, $k$, $l$, $u$ with $|t|\le N^{1/4}$, $|k|\le N^{1/4}/|t|$, $(k,l) = 1$ and
$|k|\sqrt{Q}/(\sqrt{|D_K|}N^{1/4}) \le u \le \sqrt{Q}/|t|$, we have
\begin{equation*}
%%\label{condi}
     A_t(u,k,l) \le \left(1 + \frac{|\mathcal{S}_t|/\norm(k)}{Q/|t|^2}\cdot u^2 \right)X.
\end{equation*}
   Then
\begin{equation*}
%%\label{gensparse}
\sum\limits_{q\in \mathcal{S}}  \sum\limits_{\substack{a \bmod{q}\\ (a,q)=1}}   \left|\sum\limits_{\substack{n\in \mathcal{O}_K\\ \norm(n)\le N}}
a_n \cdot \widetilde{e}_K\left(\frac{na}{q}\right)\right|^2 \ll \left(N + QXN^{\varepsilon}\left(\sqrt{N} + |\mathcal{S}|\right)\right)\sum\limits_{\substack{n\in \mathcal{O}_K\\ \norm(n)\le N}} |a_n|^2.
\end{equation*}
\end{theorem}

    Theorem \ref{thirdthm} allows us to generalize \eqref{sqrimprv} in the next theorem to all imaginary quadratic number field of class number one.
\begin{theorem}\label{squaremoduliK}
Let $K$ be any imaginary quadratic number field of class number one.  We have
 \begin{equation*}
 %%\label{squaremods}
\sum\limits_{\substack{q\in \mathcal{O}_K\setminus\{0\}\\ \norm(q)\le Q}}
\sum\limits_{\substack{a \bmod{q^2}\\ (a,q)=1}} \left|\sum\limits_{\substack{n\in\mathcal{O}_K \\ \norm(n)\le N}}a_n\cdot \widetilde{e}_K\left(\frac{na}{q^2}\right)\right|^2
 \ll (QN)^{\varepsilon}\left(Q^3 + Q^2\sqrt{N} + N\right) \sum\limits_{\substack{n\in \mathcal{O}_K\\ \norm(n)\le N}} |a_n|^2,
 \end{equation*}
 where $\varepsilon$ is any positive constant, and the implied constant $\ll$-constant depends only on $\varepsilon$.
 \end{theorem}

  Our last result derives from Theorem \ref{secdthm} a version of the large sieve for all imaginary quadratic number field of class number one when $\mathcal{S}$ is the full set of all primes with norm $\le Q$. This result can be regarded as an analogue of the above mentioned result of Wolke \cite{Wol} for prime moduli in the classical setting.
 \begin{theorem}\label{primemoduli}
 Let $K$ be any imaginary quadratic number field of class number one. Let $Q\geq 16$, $N = Q^{1+\delta}/16$, $0 < \delta < 1$. Then
 \begin{equation*}
 \sum\limits_{
 \norm(p)\le Q}\sum\limits_{\substack{a\bmod{p}\\ (a,p)=1}}\left|\sum\limits_{\substack{n\in \mathcal{O}_K\\ \norm(n)\le N}}
a_n \cdot \widetilde{e}_K\left(\frac{na}{p}\right)\right|^2 \ll
 \frac{1}{1-\delta}\cdot \frac{Q^2\log\log Q}{\log Q}\sum\limits_{\substack{n\in \mathcal{O}_K\\ \norm(n)\le N}} |a_n|^2,
 \end{equation*}
 where $p$ runs over the primes in $\mathcal{O}_K$.
 \end{theorem}

Our proofs of Theorems \ref{secdthm}--\ref{primemoduli} are slight modifications of the proofs of Theorems 1--4 in \cite{Baier&Bansa2}, the main ingredient being a Dirichlet approximation theorem in $\mc$ using elements in $K$ (see Lemma \ref{LemmaDA}).  We shall therefore only indicate the necessary modifications in Section \ref{sec 4} and skip most of the details. \newline
    
We end the section with the following remarks.  The condition of class number one in Theorems~\ref{squaremoduliK} and~\ref{primemoduli} ensures that the ring of integers is a unique factorization domain, a requirement in the proofs of those theorems.  It would also be interesting to work out the analogues of the theorems in this paper for real quadratic fields, as the situation there is quite different (the infinite group of units, for example).

\subsection{Notations} The following notations and conventions are used throughout the paper.\\
\noindent $e(z) = \exp (2 \pi i z) = e^{2 \pi i z}$. \newline
$f =O(g)$ or $f \ll g$ means $|f| \leq cg$ for some unspecified
positive constant $c$. \newline

\section{Preliminaries}
\label{sec 2}

\subsection{Imaginary quadratic number fields}
\label{sect: Kronecker}

    Let $K$ be an imaginary quadratic number field. Then we have $K=\mq(\sqrt{d})$ with $d$ a negative and square-free rational integer. The following facts concerning $K$ can be found in \cite[Section 3.8]{iwakow}. The ring of integers $\mathcal{O}_K$ is a free $\mz$ module, $\mathcal{O}_K=\mz+\omega_K \mz$, where
\begin{align*}
   \omega_K & =\begin{cases}
\displaystyle     \frac {1}{2}(1+\sqrt{d}) \qquad & \text{if $d \equiv 1 \pmod 4$ }, \\ \\
     \sqrt{d} \qquad & \text{if $d \equiv 2, 3 \pmod 4$ }.
    \end{cases}
\end{align*}

   Note that if we write $q=q_1+q_2\omega_K$, then
\begin{align*}
\norm(q)=\begin{cases}
   \displaystyle  q^2_1+q_1q_2+\frac {1-d}{4}q^2_2 \qquad & \text{if $d \equiv 1 \pmod 4$ }, \\ \\
   \displaystyle  q^2_1-dq^2_2 \qquad & \text{if $d \equiv 2, 3 \pmod 4$ }.
    \end{cases}
\end{align*}

%%------------------------------------------------------------------------------
\subsection{Poisson Summation in number fields}
%%------------------------------------------------------------------------------
      We note the following Poisson summation formula for
   $\mathcal{O}_K$ (see the proof of \cite[Lemma 4.1]{G&Zhao}), which is itself an easy consequence of the classical Poisson summation formula in
$2$ dimensions:
\begin{align*}
   \sum_{j \in \mathcal{O}_K}f(j)=\sum_{k \in
   \mathcal{O}_K}\widetilde{f}(k), \quad with \quad
   \widetilde{f}(k)=\iint\limits_{\mr^2}f(x+y\omega_K)\widetilde{e}_K\left( -k(x+y\omega_K) \right) \dif x \dif y.
\end{align*}

    We readily derive from the above Poisson summation formula that for any $b \in \mathcal{O}_K$, $Q>0$,
\begin{align}
\label{Poisson1}
\sum\limits_{x\in \mathcal{O}_K} \widetilde{e}_K \left(b \cdot x\right)f\left(\frac{x}{\sqrt{Q}}\right)
= Q \cdot \sum\limits_{y\in -b+\mathcal{O}_K} \widetilde{f}\left(\sqrt{Q}y\right).
\end{align}

   We shall also need the following Poisson summation formula for $K$:
\begin{lemma}
\label{Poissonsum} For any Schwartz class function $W$,  we have for all
$X>0$,
\begin{align*}
%%\label{PoissonsumK}
   \sum_{\substack{ m \in \mathcal{O}_K \\ m \equiv r \bmod n}}W\left(\frac {\norm(m)}{X}\right)=\frac {X}{\norm(n)}\sum_{k \in
   \mathcal{O}_K}\widetilde{W}_K\left(\sqrt{\frac {\norm(k)X}{\norm(n)}}\right)\widetilde{e}_K\left(\frac {kr}{n}\right),
\end{align*}
   where
\begin{align*}
   \widetilde{W}_K(t) &=\iint\limits_{\mr^2}W(\norm(x+y\omega_K))\widetilde{e}_K\left(- t(x+y\omega_K)\right)\dif x \dif y, \quad t
   \geq 0.
\end{align*}
\end{lemma}

It follows from \cite[(2.15)]{G&Zhao2} that for any $j \geq 1$,
\begin{align}
\label{Wbound}
     \widetilde{W}_K (t)\ll_j  \min \{ 1, \; t^{-j} \}.
\end{align}

%%------------------------------------------------------------------------------
\subsection{Dirichlet approximation in $\mc$}
%%------------------------------------------------------------------------------

     In the proof of Theorem \ref{secdthm}, we need the following version of the Dirichlet approximation theorem in $\mc$ which enables us to approximate $z \in \comc$ by an element of $K$.  This lemma generalizes the result given in \cite[Theorem 4.5]{D&K}.
\begin{lemma} \label{LemmaDA}
  Given any $z=x+iy \in \mc$ and $N \in \mn$, there exist algebraic integers $p=p_1+p_2\omega_K, q=q_1+q_2\omega_K$ in $\mathcal{O}_K$ with $0<|q| \leq N$ such that
\begin{align}
\label{Dirichletapprox}
  \Big |z-\frac {p}{q} \Big | \leq \frac {\sqrt{|D_K|}}{|q|N}.
\end{align}
\end{lemma}
\begin{proof}
The inequality in \eqref{Dirichletapprox} can be written as
\begin{align*}
  \left| x+iy-\frac {p_1+p_2\omega_K}{q_1+q_2\omega_K} \right| \leq \frac {\sqrt{|D_K|}}{|q_1+q_2\omega_K|N}.
\end{align*}
  We recast the above inequality as
\begin{align}
\label{DATbound}
  \left| ( x+iy)(q_1+q_2\omega_K)-(p_1+p_2\omega_K) \right| \leq \frac {\sqrt{|D_K|}}{N}.
\end{align}

   We simplify the above inequality according to the value of $\omega_K$. When $\omega_K=\sqrt{d}$, we rewrite it as
\begin{align*}
  \left| \left( q_1x+\sqrt{-d}q_2y-p_1 \right)+i \left( q_1y+q_2x\sqrt{-d}-p_2\sqrt{-d} \right) \right| \leq \frac {\sqrt{|D_K|}}{N}.
\end{align*}
  Then inequality \eqref{DATbound} holds if
\begin{align*}
  \max \left\{ \left| q_1x+\sqrt{-d}q_2y-p_1 \right|, \; \left| q_1y+q_2x\sqrt{-d}-p_2\sqrt{-d} \right| \right\} \leq \frac {\sqrt{|D_K|}/\sqrt{2}}{N}.
\end{align*}

    Now, by Minkowski's linear forms theorem (see \cite[p. 67, Theorem 1.41]{P&S}), the system of inequalities:
\begin{align*}
 & \left| q_1x+\sqrt{-d}q_2y-p_1 \right| \leq \frac {\sqrt{|D_K|}/\sqrt{2}}{N}, \\
 & \left| q_1y+q_2x\sqrt{-d}-p_2\sqrt{-d} \right| \leq \frac {\sqrt{|D_K|}/\sqrt{2}}{N},\\
 & |q_1| \leq 2^{-1/2}N, \\
 & |q_2| \leq 2^{-1/2}N/\sqrt{-d}
\end{align*}
   has a non-zero solution in integers $p_1, p_2, q_1, q_2$.  Hence \eqref{Dirichletapprox} has a solution with $0<|q|=|q_1+q_2\omega_K| \leq N$. \newline

    When $\omega_K= \left( 1+\sqrt{d} \right)/2$, we rewrite inequality \eqref{DATbound} as
\begin{align*}
  \left| \left( q_1x+q_2x/2+\sqrt{-d}q_2y/2-p_1-p_2/2 \right)+i \left( q_1y+q_2x\sqrt{-d}/2+q_2y/2-p_2\sqrt{-d}/2 \right) \right| \leq \frac {\sqrt{|D_K|}}{N}.
\end{align*}
  Then inequality \eqref{DATbound} holds if we have
\begin{align*}
  \max \left\{ \left| q_1x+q_2x/2+\sqrt{-d}q_2y/2-p_1-p_2/2 \right|, \; \left| q_1y+q_2x\sqrt{-d}/2+q_2y/2-p_2\sqrt{-d}/2 \right| \right\} \leq \frac {\sqrt{|D_K|}/\sqrt{2}}{N}.
\end{align*}

    Again, it follows from Minkowski's linear forms theorem that the system of inequalities:
\begin{align*}
 & |q_1x+q_2x/2+\sqrt{-d}q_2y/2-p_1-p_2/2| \leq \frac {\sqrt{|D_K|}/\sqrt{2}}{N}, \\
 & |q_1y+q_2x\sqrt{-d}/2+q_2y/2-p_2\sqrt{-d}/2| \leq \frac {\sqrt{|D_K|}/\sqrt{2}}{N},\\
 & |q_1+q_2/2| \leq 2^{-1/2}N, \\
 & |q_2| \leq 2^{1/2}N/\sqrt{-d}
\end{align*}
  has a non-zero solution in integers $p_1, p_2, q_1, q_2$.  Hence \eqref{Dirichletapprox} has a solution with $0<|q|=|q_1+q_2\omega_K| \leq N$. This completes the proof of the lemma.
\end{proof}

\subsection{Large sieve for $\mathbb{R}^m$}
 Let $s=(s_1,s_2,..,s_m) \in \mathbb{R}^m$.  We denote the Euclidean norm of $s$ by $\| s \|_2$. Thus,
$$
\| s \|_2=\sqrt{\sum^m_{i=1}s_i^2}.
$$
    In the proof of our results, we shall also need the following two versions of the large sieve. The first one is valid for all $m$, which is established in \cite[Theorem 3]{Baier&Bansal}:
\begin{lemma} \label{ls} Let $R,N\in \mathbb{N}$, $N\ge 2$, $x_1, \ldots,x_R\in \mathbb{R}^m$ and $(a_n)_{n\in \mathbb{Z}^m}$
be any $m$-fold sequence of complex numbers. Then
$$
\sum\limits_{i=1}^R \left|\sum\limits_{\substack{n\in \mathbb{Z}^m\\ \| n \|_2\le N^{1/m}}}
a_{n} \cdot e\left(n\cdot x_i\right)\right|^2 \ll FN\sum\limits_{\substack{n\in \mathbb{Z}^m\\ \| n \|_2\le N^{1/m}}} |a_n|^2,
$$
where
$$
F=\max\limits_{1\le i\le R} \sharp \left\{j\in \{1,\ldots,R\} : \min\limits_{z\in \mathbb{Z}^m} \| x_j-x_i-z \|_2\le \frac{\sqrt{m}}{N^{1/m}} \right\}.
$$
\end{lemma}

   The next one is specific to the case $m=2$, which is established in \cite[Theorem 5]{Baier&Bansa2}:
\begin{lemma}\label{lsR2}
 Let $R, N\in \mathbb{N}$, $N\geq 2$, $x_1,\ldots,x_R \in \mathbb{R}^2$ and $(a_n)_{n\in \mathbb{Z}^2}$ be any double sequence of complex numbers. Suppose that $0<\Delta \le 1/2$. Set
 \begin{equation*}
   K_0(\Delta) = \sup\limits_{\alpha\in \mathbb{R}^2} \left|\left\{r\in\{1,2,...,R\} : \min\limits_{z\in\mathbb{Z}^2} \|x_r-\alpha-z\|_2\le \Delta^{1/2}\right\}\right|.
\end{equation*}
   Then
\begin{equation*}
\sum\limits_{i=1}^{R}\left|\sum\limits_{\substack{n\in \mathbb{Z}^2\\ \|n \|_2 \le N^{1/2}}}a_n\cdot e(n\cdot x_i)\right|^2 \ll K_0(\Delta)(N+ \Delta^{-1})\sum\limits_{\substack{n\in \mathbb{Z}^2 \\ \| n \|_2\le N^{1/2}}} |a_n|^2.
\end{equation*}
\end{lemma}

\section{Proof of Theorem \ref{mainthm}}

\subsection{A general treatment}
   Let
\begin{align}
\label{Tdef}
T=\sum\limits_{\substack{q\in S\\ Q/2<\norm(q)\le Q}} \
\sum\limits_{\substack{r \bmod{q}\\ (r,q)=1}} \left|\sum\limits_{\substack{n\in \mathcal{O}_K \\ \norm(n)\le N}}
a_n \cdot \widetilde{e}_K\left (\frac{nr}{q}\right)\right|^2.
\end{align}
  Here $S$ is an arbitrary multiset of elements of $\mathcal{O}_K \setminus \{0\}$.  We shall first estimate $T$ in general and later restrict $S$ to the set of $k$-th powers. \newline

  We now write $n=s+t\omega_K$ and note that for any $z \in \mc$, we have
\begin{align*}
   \widetilde{e}_K(zn)=e\left ( \left (\frac {z-\overline{z}}{\sqrt{D_K}}, \frac {z\omega_K-\overline{z}\overline{\omega}_K}{\sqrt{D_K}}\right ) \cdot \left (s, t \right ) \right ).
\end{align*}

  We apply this to rewrite $T$ as
\begin{align}
\label{T}
\begin{split}
T= & \sum\limits_{\substack{q\in S\\ Q/2<\norm(q)\le Q}} \ \sum\limits_{\substack{r \bmod{q}\\ (r,q)=1}} \left|\sum\limits_{\substack{n\in \mathcal{O}_K\\ \norm(n)\le N}}
a_n \cdot e\left(\left(\frac{r/q-\overline{r/q}}{\sqrt{D_K}},\frac{r/q\omega_K-\overline{r/q}\overline{\omega}_K}{\sqrt{D_K}}\right)\cdot (s,t)\right)\right|^2 \\
=& \sum\limits_{\substack{q\in S\\ Q/2<\norm(q)\le Q}} \ \sum\limits_{\substack{r \bmod{q}\\ (r,q)=1}} \left|\sum\limits_{\substack{(s,t) \in \mathbb{Z}^2 \\ \|(s,t) \|_2 \le \sqrt{N}}}
a_{s,t} \cdot e\left(\left(\frac{r/q-\overline{r/q}}{\sqrt{D_K}},\frac{r/q \omega_K-\overline{r/q}\overline{\omega}_K}{\sqrt{D_K}}\right)\cdot (s,t)\right)\right|^2,
\end{split}
\end{align}
   where we define (note that $n=s+t\omega_K$)
\begin{align*}
   a_{s,t} & =\begin{cases}
     a_n \qquad & \text{if $\norm(n) \leq N$ }, \\
   0 \qquad & \text{otherwise}.
    \end{cases}
\end{align*}
   The second equality in \eqref{T} then follows by observing that $\norm(n) \leq N$ implies that $\|(s,t) \|_2 \le \sqrt{N}$ . \newline

   We now apply Lemma \ref{ls} with $m=2$ to \eqref{T} to see that we have
\begin{align}
\label{Tbound}
   T\ll EN\sum\limits_{\substack{n\in \mathcal{O}_K\\ \norm(n)\le N}} |a_n|^2,
\end{align}
where
\begin{equation*}
\begin{split}
E= \max\limits_{r_1,q_1} & \sharp \Bigg\{(r_2,q_2), r_2 \pmod {q_2}, (r_2, q_2)=1: \\
& \min\limits_{z\in \mathbb{Z}^2}
\left\| \left(\frac{r_2/q_2-\overline{r_2/q_2}}{\sqrt{D_K}},\frac{r_2/q_2 \omega_K-\overline{r_2/q_2}\overline{\omega}_K}{\sqrt{D_K}}\right)-
\left(\frac{r_1/q_1-\overline{r_1/q_1}}{\sqrt{D_K}},\frac{r_1/q_1 \omega_K-\overline{r_1/q_1}\overline{\omega}_K}{\sqrt{D_K}}\right) -z \right\|_2^2
\le \frac{1}{N} \Bigg\}.
\end{split}
\end{equation*}
   Here we make the conventions that for $j=1,2$,
$$
q_j\in S,\quad Q/2<\norm(q_j)\le Q.
$$

   We note that by writing $q=q_1+q_2\omega_K$ with $q_1, q_2 \in \mr$, we have (noting that $\omega_K-\overline{\omega}_K=\sqrt{D_K}$)
\begin{align}
\label{realandimagpart}
   \frac {q-\overline{q}}{\sqrt{D_K}}=q_2, \quad \frac {q\omega_K-\overline{q} \overline{\omega}_K }{\sqrt{D_K}}=q_1+q_2(\omega_K+\overline{\omega}_K).
\end{align}

   From this we see that by writing $z=(z_1, z_2)$ and replacing $r_2$ by
\begin{align*}
\begin{cases}
\displaystyle     r_2-q_2(z_2+z_1\omega_K) \qquad & \text{if \quad $\omega_K=\sqrt{d}$ }, \\
     r_2-q_2(z_2-z_1+z_1\omega_K) \qquad & \text{if \quad $\omega_K=\frac {1+\sqrt{d}}{2}$ },
    \end{cases}
\end{align*}
  we can drop the requirement for $r_2$ to run over the set of residue classes modulo $q_2$ so that $r_2$ is now regarded as an algebraic integer co-prime
  to $q_2$. We thus deduce that $E$ is majorized by
\begin{equation*}
\begin{split}
\max\limits_{r_1,q_1} \sharp \Bigg\{(r_2,q_2), (r_2, q_2)=1 : \left\| \left(\frac{r_2/q_2-\overline{r_2/q_2}}{\sqrt{D_K}},\frac{r_2/q_2 \omega_K-\overline{r_2/q_2}\overline{\omega}_K}{\sqrt{D_K}}\right)-
\left(\frac{r_1/q_1-\overline{r_1/q_1}}{\sqrt{D_K}},\frac{r_1/q_1 \omega_K-\overline{r_1/q_1}\overline{\omega}_K}{\sqrt{D_K}}\right) \right\|_2^2
\le \frac{1}{N} \Bigg\}.
\end{split}
\end{equation*}

   It follows from \eqref{realandimagpart} that
\begin{align}
\label{L2bound}
  \left\| \left( \frac {z-\overline{z}}{\sqrt{D_K}}, \frac {z\omega_K-\overline{z}\overline{\omega}_K}{\sqrt{D_K}} \right) \right\|_2^2 \geq |d|^{-1}\norm(z).
\end{align}
  We then obtain from \eqref{L2bound} that
\begin{equation*}
\begin{split}
E \leq  & \max\limits_{\substack{q_1\in S\\ Q/2<\norm(q_1)\le Q\\ (r_1,q_1)=1}}
\sharp \Bigg\{(r_2,q_2),(r_2, q_2)=1 : q_2\in S,\ Q/2<\norm(q_2)\le Q,\ (r_2,q_2)=1, |d|^{-1} \norm \left(\frac{r_2}{q_2}-\frac{r_1}{q_1} \right)
\le N^{-1}\Bigg\}.
\end{split}
\end{equation*}

   We now choose two Schwartz class functions $\Phi_{i}$ for $i=1,2$ satisfying
$\Phi_{i}(x)\gg 1$ when $|x|\le 1$. We can let $\Phi_{1}$ to be arbitrary and we shall fix $\Phi_{2}$ later. We further define
$$
\Psi_i=\Phi_i \circ \norm \quad \mbox{for} \; i=1, \; 2.
$$
  Using these notations together with the observation that
$$
|d|^{-1} \norm \left( \frac{r_2}{q_2}-\frac{r_1}{q_1} \right)
\le N^{-1} \Longleftrightarrow \norm\left(r_1q_2-r_2q_1\right)
\le \frac{|d|\norm(q_1)\norm(q_2)}{N},
$$
  we infer that
\begin{equation} \label{Ktransform}
\begin{split}
E\le & \max\limits_{\substack{q_1\in S\\ Q/2<\norm(q_1)\le Q\\ (r_1,q_1)=1}} \
\sum\limits_{\substack{b\in \mathcal{O}_K \\ \norm(b)\le |d|\norm(q_1)\norm(q_2)/N}} \
\sum\limits_{\substack{q_2\in S,\\ Q/2<\norm(q_2)\le Q\\  b\equiv r_1q_2 \bmod{q_1}}} 1 \\
\ll & \max\limits_{\substack{q_1\in S\\ Q/2<\norm(q_1)\le Q\\ (r_1,q_1)=1}} \
\sum\limits_{b\in \mathcal{O}_K} \Phi_1\left(\norm\left(\frac{b\sqrt{N}}{q_1\sqrt{|d|Q}}\right)\right)
\sum\limits_{\substack{q_2\in S,\\ b\equiv  r_1q_2\bmod{q_1}}} \Phi_2\left(\norm\left(\frac{q_2}{\sqrt{Q}}\right)\right)\\
= & \max\limits_{\substack{q_1\in S\\ Q/2<\norm(q_1)\le Q\\ (r_1,q_1)=1}} \
\sum\limits_{q_2\in S} \Phi_2\left(\norm\left(\frac{q_2}{\sqrt{Q}}\right)\right)\cdot
\sum\limits_{b\equiv r_1q_2 \bmod{q_1}} \Phi_1\left(\norm\left(\frac{b\sqrt{N}}{q_1\sqrt{|d|Q}}\right)\right) \\
= & \max\limits_{\substack{q_1\in S\\ Q/2<\norm(q_1)\le Q\\ (r_1,q_1)=1}} \
\sum\limits_{q_2\in S} \Psi_2\left(\frac{q_2}{\sqrt{Q}}\right)\cdot
\sum\limits_{b\equiv r_1q_2 \bmod{q_1}} \Psi_1\left(\frac{b\sqrt{N}}{q_1\sqrt{|d|Q}}\right).
\end{split}
\end{equation}

Applying the Poisson summation formula, Lemma \ref{Poissonsum}, yields
\begin{equation} \label{afterpoisson}
\begin{split}
\sum\limits_{b\equiv r_1q_2 \bmod{q_1}} \Psi_1\left(\frac{b\sqrt{N}}{q_1\sqrt{|d|Q}}\right)
= & \frac{|d|Q}{N}\cdot \sum\limits_{j\in \mathcal{O}_K}\widetilde{e}_K \left ( \frac {jr_1q_2}{q_1} \right)
\widetilde{\Phi}_{1,K} \left(\sqrt{\frac {\norm(j)|d|Q}{N}}\right).
\end{split}
\end{equation}

So from \eqref{Ktransform} and \eqref{afterpoisson}, we get that
\begin{equation} \label{Kgen1}
E \ll \frac{Q}{N}\cdot\max\limits_{\substack{q_1\in S\\ Q/2<\norm(q_1)\le Q\\ (r_1,q_1)=1}} \
\sum\limits_{j\in \mathcal{O}_K} \widetilde{\Phi}_{1,K}\left(\sqrt{\frac {\norm(j)|d|Q}{N}}\right)
\sum\limits_{q_2\in S} \Psi_2\left(\frac{q_2}{\sqrt{Q}}\right)\cdot
\widetilde{e}_K \left ( \frac {jr_1q_2}{q_1} \right).
\end{equation}

\subsection{Weyl differencing}
Now we take $S$ as the set of non-zero $k$-th powers in $\mathcal{O}_K$. We write $Q_0=Q^{1/k}$
and replace $q_i$ by $q_i^k$ ($i=1,2$).  In what follows, we assume that $Q_0>N^{1/(2k)}$ for otherwise the desired result follows from
\eqref{Huxley} upon extending the set of moduli to all non-zero integers in $\mathcal{O}_K$. We further note that \eqref{Wbound} gives
$$
\sum\limits_{q_2\in \mathcal{O}_K} \Psi_2\left(\frac{q_2^k}{Q^{k/2}_0}\right) \ll Q_0.
$$
  We use the above estimate to bound the contribution of $j=0$ on the right-hand side of \eqref{Kgen1} to see that
\begin{equation} \label{generalKbound}
\begin{split}
 E \ll & \frac{Q_0^k}{N}\cdot \max\limits_{\substack{Q_0/\sqrt[k]{2}<\norm(q_1)\le Q_0\\ (r_1,q_1)=1}}\sum\limits_{j\in \mathcal{O}_K} \widetilde{\Phi}_{1,K}\left(\sqrt{\frac {\norm(j)|d|Q_0^k}{N}}\right) \cdot
\sum\limits_{q_2\in \mathcal{O}_K} \Psi_2\left(\frac{q_2^k}{Q^{k/2}_0}\right)\cdot
\widetilde{e}_K \left ( \frac {jr_1q^k_2}{q^k_1} \right) \\
\ll & \frac{Q_0^{k+1}}{N}+ \frac{Q_0^k}{N}\cdot \max\limits_{\substack{Q_0/\sqrt[k]{2}<\norm(q_1)\le Q_0\\ (r_1,q_1)=1}}
\sum\limits_{j\in \mathcal{O}_K \setminus \{0\}} \widetilde{\Phi}_{1,K}\left(\sqrt{\frac {\norm(j)|d|Q_0^k}{N}}\right) \cdot
\left| S_k\left(q_1,r_1,j\right)\right|,
\end{split}
\end{equation}
where
$$
S_k\left(q_1,r_1,j\right)=\sum\limits_{q_2\in \mathcal{O}_K} \Psi_2\left(\frac{q_2^k}{Q^{k/2}_0}\right)\cdot
\widetilde{e}_K \left ( \frac {jr_1q^k_2}{q^k_1} \right).
$$

Multiplying out the square and setting $\alpha_1=q_2-q$, we obtain
\begin{equation*}
\begin{split}
 \left|S_k\left(q_1,r_1,j\right) \right|^2 =  & \sum\limits_{q_2,q\in \mathcal{O}_K } \Psi_2\left(\frac{q_2^k}{Q_0^{k/2}}\right) \cdot \Psi_2\left(\frac{q^k}{Q_0^{k/2}}\right) \cdot
\widetilde{e}_K \left( \frac{jr_1}{q_1^k}\cdot(q_2^k-q^k)  \right)\\
 =  & \sum\limits_{\alpha_1,q\in \mathcal{O}_K }\Psi_2\left(\frac{q^k}{Q_0^{k/2}}\right) \cdot \Psi_2\left(\frac{(\alpha_1+q)^k}{Q_0^{k/2}}\right)
 \cdot \widetilde{e}_K \left( \frac{jr_1}{q_1^k}\cdot\left((\alpha_1 + q)^k-q^k  \right) \right).
 \end{split}
 \end{equation*}
We observe that the contribution of  $q, \alpha$'s with $\norm\left(q \right)$, $\norm\left(\alpha_1+ q \right)>Q_0^{1+\varepsilon}$ is negligible, it thus follows that the contribution of $\alpha_1$'s with $\norm\left(\alpha_1\right)>Q_0^{1+\varepsilon}$ is negligible. We
write
$$
P_{k-1,\alpha_1}(q) = (\alpha_1 + q)^k - q^k = \binom{k}{1} \cdot \alpha_1 q^{k-1} + \binom{k}{2}\cdot \alpha_1^2q^{k-2} +\cdots +
\binom{k}{k}\cdot \alpha_1^k,
$$
and get that
\begin{equation*}
\left|S_k\left(q_1,r_1,j\right) \right|^2 \ll \left|\sum\limits_{\norm(\alpha_1)\le Q_0^{1+\varepsilon}} S_{k-1}\left(q_1,r_1,j,\alpha_1\right)\right|,
 \end{equation*}
 where
 $$
 S_{k-1}\left(q_1,r_1,j,\alpha_1\right):=\sum\limits_{q\in \mathcal{O}_K }
\Psi_2\left(\frac{q^k}{Q_0^{k/2}}\right) \cdot \Psi_2\left(\frac{(\alpha_1+q)^k}{Q_0^{k/2}}\right)
\cdot \widetilde{e}_K \left(\frac{jr_1}{q_1^k}\cdot P_{k-1,\alpha_1}(q) \right).
 $$
 If $k > 2$, the Cauchy-Schwarz inequality gives
 \begin{equation*}
 \left| S_k\left(q_1,r_1,j\right) \right|^4 \ll
 Q_0^{1+ \varepsilon}\sum\limits_{\substack{\alpha_1 \in \mathcal{O}_K  \\ \norm(\alpha_1) \le Q_0^{1+\varepsilon}}}
 \left| S_{k-1}\left(q_1,r_1,j,\alpha_1\right)  \right|^2.
 \end{equation*}
Multiplying out the square, changing variables and truncating the resulting sums in a similar way as above, we obtain
\begin{equation*}
\begin{split}
 & \left| S_{k-1}\left(q_1,r_1,j,\alpha_1\right)
 \right|^2\\
 & \ll  \left| \sum\limits_{\substack{\alpha_2\in \mathcal{O}_K \\ \norm(\alpha_2)\le Q_0^{1+\varepsilon}}} \sum\limits_{q \in \mathcal{O}_K }
 \Psi_2\left(\frac{q^k}{Q_0^{k/2}}\right)  \Psi_2\left(\frac{(\alpha_1 + q)^k}{Q_0^{{k/2}}}\right)
 \Psi_2\left(\frac{(\alpha_2 + q)^k}{Q_0^{k/2}}\right) \Psi_2\left(\frac{(\alpha_1 + \alpha_2 + q)^k}{Q_0^{{k/2}}}\right)\cdot
   \widetilde{e}_K\left(\frac{jr_1}{q_1^k}\cdot P_{k-2,\alpha_1,\alpha_2}(q)\right) \right|,
 \end{split}
 \end{equation*}
where
\begin{equation*}
P_{k-2,\alpha_1,\alpha_2}(q) = k(k-1)\alpha_1 \alpha_2 q^{k-2} + \cdots
\end{equation*}
is a polynomial of degree $k-2$ in $q$. We continue this process of repeated use of Cauchy-Schwarz and differencing until we have reached
a linear polynomial so that
\begin{equation*}
%%\label{again}
\begin{split}
& \left|S_k\left(q_1,r_1,j\right) \right|^{\kappa} \ll  Q_0^{\kappa-k+ \varepsilon} \sum\limits_{\substack{\alpha \in \mathcal{O}_K^{k-1} \\
\norm\left(\alpha_1\right),...,\norm\left(\alpha_{k-1}\right)\le Q_0^{1+\varepsilon}}} \left|
\sum\limits_{q \in \mathcal{O}_K } \prod_{u\in \{0,1\}^{k-1}}
\Psi_2\Bigg(\frac{(u\cdot \alpha+q)^k}{Q_0^{k/2}}\Bigg) \cdot
\widetilde{e}_K\left(\frac{jr_1}{q_1^k}\cdot P_{1,\alpha}(q)\right) \right|,
 \end{split}
\end{equation*}
where $\kappa=2^{k-1}$, $\alpha=\left(\alpha_1,...,\alpha_{k-1}\right)$, $u=\left(u_1,...,u_k\right)$, $u\cdot \alpha$ is the standard inner product and
\[ P_{1,\alpha}(q) = k! \alpha_1\cdots \alpha_{k-1}\cdot \left(q + \frac{1}{2}\cdot \left(\alpha_1+\cdots +\alpha_{k-1}\right)\right). \]

\subsection{Poisson summation}

We now specify our choice of $\Psi_2$ by setting
$$
\Phi_2(t)=\exp\left(-\frac{\pi}{\kappa}\cdot \sqrt[k]{|t|}\right) \quad \mbox{so that} \quad
\Psi_2(z)=\Phi_2(\norm(z))=\exp\left(-\frac{\pi}{\kappa}\cdot \sqrt[k]{\norm(z)}\right).
$$

We further set
$$
g(z)=\prod_{u\in \{0,1\}^{k-1}}
\Psi_2\left(\left(z + \frac{u\cdot \alpha}{\sqrt{Q_0}}\right)^k\right).
$$
Then by taking
$$
 b=\frac{k! \alpha_1\cdots \alpha_{k-1}jr_1}{q_1^k},
$$
we obtain that
\begin{equation*}
%%\label{backtor3}
\begin{split}
& \Bigg | \sum\limits_{q\in \mathcal{O}_K} \prod_{u\in \{0,1\}^{k-1}}
\Psi_2\Bigg(\frac{(u\cdot \alpha+q)^k}{Q_0^{k/2}}\Bigg) \cdot
\widetilde{e}_K\left(\frac{jr_1}{q_1^k}\cdot P_{1,\alpha}(q) \right) \Bigg |
= \Bigg |\sum\limits_{x\in \mathcal{O}_K} \widetilde{e}_K\left(b\cdot x\right)g\left(\frac{x}{\sqrt{Q_0}}\right)\Bigg | .
\end{split}
\end{equation*}
Applying \eqref{Poisson1} gives
\[ \sum\limits_{x\in \mathcal{O}_K} \widetilde{e}_K \left(b \cdot x\right)g\left(\frac{x}{\sqrt{Q_0}}\right)
= Q_0\cdot \sum\limits_{y\in -b+\mathcal{O}_K} \widetilde{g}\left(\sqrt{Q_0}y\right). \]

It follows that
\begin{equation} \label{newpoiss1}
\begin{split}
& \left|S_k\left(q_1,r_1,j\right) \right|^{\kappa} \ll  Q_0^{\kappa-k+1+ \varepsilon}
\sum\limits_{\substack{\alpha \in \mathcal{O}_K^{k-1} \\
\norm\left(\alpha_1\right),...,\norm\left(\alpha_{k-1}\right)\le Q_0^{1+\varepsilon}}}
\sum\limits_{\beta\in \mathcal{O}_K}  \widetilde{g}\left(\sqrt{Q_0}\cdot
\left(\beta-\frac{k! \alpha_1\cdots \alpha_{k-1}jr_1}{q_1^k}\right)\right).
\end{split}
\end{equation}

   To compute the Fourier transform of $g(z)$, we note that
\begin{align*}
g(z)=\begin{cases}
   \displaystyle  \exp\left(-\frac{\pi}{\kappa}\cdot \sum\limits_{u\in \{0,1\}^{k-1}} \left(
z_{1,u}^2+z_{1,u}z_{2,u}+\frac {1-d}{4}z_{2,u}^2\right)\right) \qquad & \text{if $d \equiv 1 \pmod 4$ }, \\ \\
   \displaystyle  \exp\left(-\frac{\pi}{\kappa}\cdot \sum\limits_{u\in \{0,1\}^{k-1}} \left( z_{1,u}^2-dz_{2,u}^2\right) \right ) \qquad & \text{if $d \equiv 2, 3 \pmod 4$ }.
    \end{cases}
\end{align*}
where we write
\begin{align*}
z=z_1+z_2\omega_K, \quad  \alpha= \alpha^{(1)}+ \alpha^{(2)}\omega_K, \quad z_{i,u}=z_i+\frac{u\cdot \alpha^{(i)}}{\sqrt{Q_0}}, \quad i=1, \; 2.
\end{align*}

Completing the squares, we deduce that
$$
g(z)=\exp\left(-\frac{\pi}{\kappa Q_0}\cdot
\left( \sum\limits_{u\in \{0,1\}^{k-1}} \norm (u\cdot \alpha) - \frac 1{\kappa}
  \norm \left( \sum\limits_{u\in \{0,1\}^{k-1}} u\cdot \alpha \right)\right) \right ) \cdot
\exp\left(-\pi\norm \left(z+\frac{\sum\limits_{v=1}^{k-1}\alpha_v^{(i)}}{2\sqrt{Q_0}}\right)^2\right).
$$
  A direct computation shows that the Fourier transform of $g(z)$ is
\begin{equation}
\label{gfourier}
\begin{split}
\widetilde{g}(z)= & A\exp\left(-\frac{\pi}{\kappa Q_0}\cdot
\left( \sum\limits_{u\in \{0,1\}^{k-1}} \norm (u\cdot \alpha) - \frac 1{\kappa}
  \norm \left( \sum\limits_{u\in \{0,1\}^{k-1}} u\cdot \alpha \right)\right) \right ) \cdot \widetilde{e}_K \left(\frac {\sum\limits_{v=1}^{k-1}\alpha_v^{(i)} z}{2\sqrt{Q_0}} \right)\cdot
\exp\left(-\frac {\pi}{|d|} \norm(z)\right),
\end{split}
\end{equation}
where $A$ is some constant whose value depends only on $d$. \newline

 As a direct consequence of the triangle inequality for norms, we note that
\begin{align*}
   \sum\limits_{u\in \{0,1\}^{k-1}} \norm (u\cdot \alpha) - \frac 1{\kappa}
  \norm \left( \sum\limits_{u\in \{0,1\}^{k-1}} u\cdot \alpha \right)  \geq 0.
\end{align*}
  We then deduce by plugging \eqref{gfourier} into \eqref{newpoiss1} that
\begin{equation} \label{finalafterps1}
\begin{split}
& \left|S_k\left(q_1,r_1,j\right) \right|^{\kappa} \ll  Q_0^{\kappa-k+1+\varepsilon}
\sum\limits_{\substack{\alpha \in \mathcal{O}_K^{k-1} \\
\norm\left(\alpha_1\right),...,\norm\left(\alpha_{k-1}\right)\le Q_0^{1+\varepsilon}}}
\sum\limits_{\beta\in \mathcal{O}_K} \exp\left(-\frac {\pi Q_0}{|d|}\norm
\left(\beta-\frac{k!\alpha_1\cdots \alpha_{k-1}j r_1}{q_1^k}\right)\right).
\end{split}
\end{equation}

\subsection{Counting}
To bound the sum in the maximum in \eqref{generalKbound}, we first note that by \eqref{Wbound},
\begin{equation} \label{countbeg}
\begin{split}
\sum\limits_{j\in \mathcal{O}_K \setminus \{0\}} \widetilde{\Phi}_{1,K} \left(\sqrt{\frac {\norm(j)|d|Q_0^k}{N}}\right) \cdot
\left| S_k\left(q_1,r_1,j\right)\right| \ll & 1+ \sum\limits_{\substack{j\in \mathcal{O}_K \setminus \{0\}\\
\norm(j)\le NQ_0^{\varepsilon-k}}}
\left| S_k\left(q_1,r_1,j\right)\right|\\
\ll & 1+ \left(\frac{N}{Q_0^{k-\varepsilon}}\right)^{1-1/\kappa} \left(\sum\limits_{\substack{j\in \mathcal{O}_K \setminus \{0\}\\
\norm(j)\le NQ_0^{\varepsilon-k}}}
\left| S_k\left(q_1,r_1,j\right)\right|^{\kappa}\right)^{1/\kappa},
\end{split}
\end{equation}
where the second line follows from  H\"older's inequality. Using \eqref{finalafterps1} and taking into account that
the contributions of $\beta$'s with
$$
\norm\left(\beta-\frac{k!\alpha_1\cdots \alpha_{k-1}j r_1}{q_1^k}\right) > Q_0^{\varepsilon-1}
$$
is negligible, we arrive at
\begin{equation*}
 \sum\limits_{\substack{j\in \mathcal{O}_K \setminus \{0\}\\
\norm(j)\le NQ_0^{\varepsilon-k}}}
\left| S_k\left(q_1,r_1,j\right)\right|^{\kappa}
\ll
Q_0^{\kappa-k+1+ \varepsilon}
\sum\limits_{\substack{j\in \mathcal{O}_K \setminus \{0\}\\
\norm(j)\le NQ_0^{\varepsilon-k}}} \
\sum\limits_{\substack{\alpha \in \mathcal{O}_K^{k-1} \\
\norm\left(\alpha_1\right),...,\norm\left(\alpha_{k-1}\right)\le Q_0^{1+\varepsilon}}} \
\sum\limits_{\substack{\beta\in \mathcal{O}_K \\
\norm\left(\beta-k!\alpha_1\cdots \alpha_{k-1}j r_1/q_1^k\right)\le Q_0^{\varepsilon-1}}} 1.
\end{equation*}

  Writing $d=k!\alpha_1\cdots \alpha_{k-1}j$ and noting that the number of divisors of $d\in \mathcal{O}_K\setminus \{0\}$ is bounded by
$O\left(\norm(d)^{\varepsilon}\right)$, we obtain
\begin{equation} \label{super1}
\begin{split}
 \sum\limits_{\substack{j\in \mathcal{O}_K \setminus \{0\}\\
\norm(j)\le NQ_0^{\varepsilon-k}}}
\left| S_k\left(q_1,r_1,j\right)\right|^{\kappa}
\ll & Q_0^{\kappa-k+1 +\varepsilon}\cdot\sum\limits_{\substack{j\in \mathcal{O}_K \setminus \{0\}\\
\norm(j)\le NQ_0^{\varepsilon-k}}} \
\sum\limits_{\substack{\alpha \in \mathcal{O}_K^{k-1} \\
\norm\left(\alpha_1\right),...,\norm\left(\alpha_{k-1}\right)\le Q_0^{1+\varepsilon}}} \
\sum\limits_{\substack{\beta, d \in \mathcal{O}_K \\
\norm\left(\beta-d r_1/q_1^k\right)\le Q_0^{\varepsilon-1}}} 1\\
\ll & (NQ_0)^{(k+1)\varepsilon} \cdot Q_0^{\kappa-k+1}\cdot \Bigg(\frac{N}{Q_0^{2}}+ \sum\limits_{\substack{d\in \mathcal{O}_K \setminus\{0\}\\
\norm(d)\le (k!)^2NQ_0^{k\varepsilon-1}}} \ \sum\limits_{\substack{\beta \in \mathcal{O}_K \\
\norm\left(\beta-d r_1/q_1^k\right)\le Q_0^{\varepsilon-1}}} 1\Bigg)\\
\ll & (NQ_0)^{(k+1)\varepsilon} \cdot \Big(NQ_0^{\kappa-k-1}+ Q_0^{\kappa-k+1}\cdot
\sum\limits_{\substack{l\in \mathcal{O}_K\\ \norm(l/q_1^k) \le Q_0^{\varepsilon-1}}}
\sum\limits_{\substack{\norm(d)\le (k!)^2NQ_0^{k\varepsilon-1}\\ d\equiv l\overline{r}_1 \bmod{q_1^k}}} 1\Bigg).
\end{split}
\end{equation}

  Observing that the number of residue classes modulo $q_1^{k}$ is $\norm(q_1^k)\le Q_0^k$, we get
\begin{equation} \label{res1}
\sum\limits_{\substack{\norm(d)\le (k!)^2NQ_0^{k\varepsilon-1}\\ d\equiv l\overline{r}_1 \bmod{q_1^k}}} 1 \ll
1+\frac{N}{Q_0^{k+1-k\varepsilon}}.
\end{equation}
  Note that we also have
\begin{equation} \label{lat1}
\sum\limits_{\substack{l\in \mathcal{O}_K\\ \norm(l/q_1^k) \le Q_0^{\varepsilon-1}}} 1 \le
\sum\limits_{\substack{l\in \mathcal{O}_K\\ \norm(l) \le Q_0^{\varepsilon+k-1}}} 1\ll Q_0^{k-1+2\varepsilon}.
\end{equation}
Combining \eqref{super1}, \eqref{res1} and \eqref{lat1}, we see that
\begin{equation} \label{comb3}
\sum\limits_{\substack{j\in \mathcal{O}_K \setminus \{0\}\\
\norm(j)\le NQ_0^{\varepsilon-k}}}
\left| S_k\left(q_1,r_1,j\right)\right|^{\kappa}
\ll (Q_0N)^{(2k+3)\varepsilon}\left(Q_0^{\kappa}+ NQ_0^{\kappa-k-1}\right).
\end{equation}

  It then follows from \eqref{generalKbound}, \eqref{countbeg} and \eqref{comb3} that we have
\begin{equation} \label{comb4}
 E \ll \frac{Q_0^{k+1}}{N}+(Q_0N)^{\varepsilon}\left(\frac{Q_0^{1+k/\kappa}}{N^{1/\kappa}}+Q_0^{1-1/\kappa}\right).
\end{equation}

The assertion of Theorem \ref{mainthm} now follows readily from \eqref{Tbound} and \eqref{comb4} by dividing the moduli into dyadic intervals and replacing $Q_0$ by $Q$.

\section{Proofs of Theorems \ref{secdthm}--\ref{primemoduli}}
\label{sec 4}

   We define $U$ as the way we define $T$ in \eqref{Tdef}, except that we remove the condition that $Q/2<\norm(q)\le Q $. We aim to estimate $U$ by first rewriting it as we did in \eqref{T} (with $n=s+t\omega_K$)
\begin{equation*}
%%\label{U}
U=\sum\limits_{q\in \mathcal{S}} \sum\limits_{\substack{r \bmod{q}\\ (r,q)=1}} \left|\sum\limits_{\substack{n\in \mathcal{O}_K\\ \norm(n)\le N}}
a_n \cdot e\left(\left(\frac{r/q-\overline{r/q}}{\sqrt{D_K}},\frac{r/q\omega_K-\overline{r/q}\overline{\omega}_K}{\sqrt{D_K}}\right)\cdot (s,t)\right)\right|^2 .
\end{equation*}

   To bound $U$, we employ Lemma \ref{lsR2} to see that
\begin{align*}
%%\label{boundforU}
U\ll K(\Delta)(N+\Delta^{-1})\sum\limits_{\substack{n\in \mathcal{O}_K\\ \norm(n)\le N}} |a_n|^2,
\end{align*}
where
\begin{align*}
 K(\Delta) = & \sup\limits_{\alpha\in\mathbb{R}^2}
 \Bigg|\Bigg\{  (r,q)\in \mathcal{O}_K\times \mathcal{S} : (r,q)=1, \min\limits_{z\in \mathbb{Z}^2}
 \left\| \left(\frac{r/q-\overline{r/q}}{\sqrt{D_K}},\frac{r/q\omega_K-\overline{r/q}\overline{\omega}_K}{\sqrt{D_K}}\right)-\alpha -z\right\|_2\le \Delta^{1/2}\Bigg\} \Bigg| \\
   = & \sup\limits_{\alpha\in\mathbb{R}^2}
 \Bigg|\Bigg\{  (r,q)\in \mathcal{O}_K\times \mathcal{S} : (r,q)=1,
 \left\| \left(\frac{r/q-\overline{r/q}}{\sqrt{D_K}},\frac{r/q\omega_K-\overline{r/q}\overline{\omega}_K}{\sqrt{D_K}}\right)-\alpha \right\|_2\le \Delta^{1/2}\Bigg\} \Bigg|.
\end{align*}

   If we write $r/q=q_1+q_2\omega_K$, $\alpha=(\alpha_1, \alpha_2)$, then it follows from \eqref{realandimagpart} that
\begin{align*}
& \left\|\left(\frac{r/q-\overline{r/q}}{\sqrt{D_K}},\frac{r/q\omega_K-\overline{r/q}\overline{\omega}_K}{\sqrt{D_K}}\right)-\alpha \right\|_2
=\sqrt{(q_2-\alpha_1)^2+(q_1+q_2(\omega_K+\overline{\omega}_K)-\alpha_2)^2}.
\end{align*}

   It is easy to see that when $\omega_K=\sqrt{d}$, we have
\begin{align}
\label{P1}
 \sqrt{(q_2-\alpha_1)^2+(q_1+q_2(\omega_K+\overline{\omega}_K)-\alpha_2)^2}
\geq & \frac 1{\sqrt{-d}}\left|\frac{r}{q}-\alpha' \right|, \; \mbox{where} \;   \alpha' =\alpha_2+\alpha_1\sqrt{-d}i.
\end{align}

   When $\omega_K=\left( 1+\sqrt{d} \right)/2$, we have
\begin{align*}
& \left( q_2-\alpha_1 \right)^2+ \left(q_1+q_2(\omega_K+\overline{\omega}_K)-\alpha_2 \right)^2 = (q_1+q_2-\alpha_2)^2+(q_2-\alpha_1)^2.
\end{align*}

    Applying the inequality $2(a^2+b^2) \geq (a+b)^2$ for any real numbers $a,b$, we see that
\begin{align*}
 2\left ( \left(q_1+q_2-\alpha_2 \right )^2+\left (\frac {q_2-\alpha_1}{2} \right )^2 \right ) \geq \left(q_1+q_2-\alpha_2-\frac {q_2-\alpha_1}{2}\right)^2=\left(q_1+\frac {q_2}{2}-\alpha_2+\frac {\alpha_1}{2} \right )^2.
\end{align*}

   Note that we also have
\begin{align*}
 \left( \frac {q_2\sqrt{-d}}{2}-\frac {\alpha_1\sqrt{-d}}{2} \right)^2 =\frac {-d}{4} \left( q_2-\alpha_1 \right)^2.
\end{align*}

   We then deduce that (note that we have $-d \geq 2$ in our case)
\begin{align*}
  \left( \frac {q_2\sqrt{-d}}{2}-\frac {\alpha_1\sqrt{-d}}{2} \right)^2 + \left(q_1+\frac {q_2}{2}-\alpha_2+\frac {\alpha_1}{2} \right )^2= & \frac {-d}{4} \left( q_2-\alpha_1 \right)^2+\left(q_1+\frac {q_2}{2}-\alpha_2+\frac {\alpha_1}{2} \right )^2  \\
\leq & (-d) \left ( \left (q_1+q_2-\alpha_2 \right )^2+\left (q_2-\alpha_1 \right )^2 \right ).
\end{align*}

  It follows that when $\omega_K=(1+\sqrt{d})/2$, we have
\begin{align}
\label{P2}
 \sqrt{(q_2-\alpha_1)^2+(q_1+q_2(\omega_K+\overline{\omega}_K)-\alpha_2)^2}
\geq & \frac 1{\sqrt{-d}}\left|\frac{r}{q}-\alpha'' \right|, \; \mbox{where} \;   \alpha'' =\alpha_2-\frac {\alpha_1}{2}+\frac {\alpha_1\sqrt{-d}}{2}i.
\end{align}

   We then conclude from \eqref{P1} and \eqref{P2} that
\[ K(\Delta) \leq  \sup\limits_{\alpha\in\mathbb{C}} P(\alpha), \]
where
\begin{align*}
%%\label{P}
\begin{split}
P(\alpha) =  \left|\left\{(a,q) \in \mathcal{O}_K\times \mathcal{S} :  (a,q) = 1,\
\left| \frac{a}{q}-\alpha\right|\le \sqrt{-d}\Delta^{1/2}\right\}\right|.
\end{split}
\end{align*}

    Replacing $\Delta$ by $\Delta/(-d)$, we may assume $d=-1$ in the definition of $P(\alpha)$. Thus we have
\begin{equation*}
P(\alpha) =  \sum\limits_{\substack{q\in \mathcal{S}, (a,q) = 1\\ a/q \in B(\alpha,\Delta^{1/2})}} 1.
  \end{equation*}

To estimate $P(\alpha)$, we approximate $\alpha$ by a suitable element of $\mathcal{O}_K$. Let
\begin{equation*}
            \tau = \frac{1}{\Delta^{1/4}}.
\end{equation*}
Then, using Lemma \ref{LemmaDA}$, \alpha$ can be written in the form
\begin{equation} \label{approxi}
\alpha = \frac{b}{r} + z, \ \text{where} \  b,r\in \mathcal{O}_K,\ (b,r) = 1,\ |z|<\frac{\sqrt{|D_K|}}{|r|\tau},\ 0<\ |r|\le\tau.
\end{equation}
Thus, it suffices to estimate $P(b/r+z)$ for all $b$, $r$, $z$ satisfying \eqref{approxi}.
We further note that, as in the case of \cite[(24)]{Baier&Bansa2}, we may assume that
\begin{equation*}
%%\label{zcond}
|z|\geq \Delta^{1/2}.
\end{equation*}

 We then deduce that
\begin{align*}
%%\label{boundforK}
K(\Delta) \ll \sup\limits_{\substack{r\in\mathcal{O}_K\\ 1\le |r|\le \tau}}\sup\limits_{\substack{b\in \mathcal{O}_K \\ (b,r) =1}}\sup\limits_{\substack{z\in\mathbb{C}\\
\Delta^{1/2}\le |z|\le \frac{\sqrt{|D_K|}}{|r|\tau}}}P\left(\frac{b}{r}+z\right).
\end{align*}

One then uses arguments similar to those in the proofs of Theorems 1--4 in \cite{Baier&Bansa2} to complete the proofs of Theorems \ref{secdthm}--\ref{primemoduli}.

\vspace{0.1in}

\noindent{\bf Acknowledgments.} P. G. is supported in part by NSFC grant 11871082 and L. Z. by the FRG grant PS43707 and the Faculty Goldstar Award PS53450. Parts of this work were done when P. G. visited the University of New South Wales (UNSW). He wishes to thank UNSW for the invitation, financial support and warm hospitality during his pleasant stay.  Finally, the authors would like to thank the anonymous referee for his/her careful reading of the paper.

\bibliography{biblio}

\begin{bibdiv}
\begin{biblist}

\bib{Ba1}{article}{
      author={Baier, S.},
       title={On the large sieve with sparse sets of moduli},
        date={2006},
     journal={J. Ramanujan Math. Soc.},
      volume={21},
       pages={279\ndash 295},
}

\bib{Baier&Bansal}{article}{
      author={Baier, S.},
      author={Bansal, A.},
       title={The large sieve with power moduli for {$\Bbb Z[i]$}},
        date={2018},
        ISSN={1793-0421},
     journal={Int. J. Number Theory},
      volume={14},
      number={10},
       pages={2737\ndash 2756},
}

\bib{Baier&Bansa2}{article}{
      author={Baier, S.},
      author={Bansal, A.},
       title={Large sieve with sparse sets of moduli for {$\Bbb Z[i]$}},
        date={2020},
     journal={Acta Arith.},
      volume={196},
      number={1},
       pages={17\ndash 34},
}

\bib{BaiLynZha}{article}{
      author={Baier, S.},
      author={Lynch, S.~B.},
      author={Zhao, L.},
       title={Elliptic curves in isogeny classes},
        date={2019},
     journal={Bull. Aust. Math. Soc.},
      volume={100},
      number={2},
       pages={225\ndash 229},
}

\bib{B&Z1}{article}{
      author={Baier, S.},
      author={Zhao, L.},
       title={Large sieve inequality with characters for powerful moduli},
        date={2005},
        ISSN={1793-0421},
     journal={Int. J. Number Theory},
      volume={1},
      number={2},
       pages={265\ndash 279},
}

\bib{B&Z}{article}{
      author={Baier, S.},
      author={Zhao, L.},
       title={An improvement for the large sieve for square moduli},
        date={2008},
        ISSN={0022-314X},
     journal={J. Number Theory},
      volume={128},
      number={1},
       pages={154\ndash 174},
}

\bib{D&K}{article}{
      author={Dodson, M.~M.},
      author={Kristensen, S.},
       title={Hausdorff dimension and diophantine approximation},
        date={2004},
     journal={in: Fractal Geometry and Applications: A Jubilee of Benoit
  Mandelbrot, part 1, Proceedings of Symposia in Pure Mathematics, Part 1,
  American Mathematical Society, Providence, RI, 2004, 305--347},
      volume={1},
      number={2},
       pages={305\ndash 347},
}

\bib{G&Zhao2}{article}{
      author={Gao, P.},
      author={Zhao, L.},
       title={One level density of low-lying zeros of families of
  {$L$}-functions},
        date={2011},
     journal={Compos. Math.},
      volume={147},
      number={1},
       pages={1\ndash 18},
}

\bib{G&Zhao}{article}{
      author={Gao, P.},
      author={Zhao, L.},
       title={Large sieve inequalities for quartic character sums},
        date={2012},
     journal={Q. J. Math.},
      volume={63},
      number={4},
       pages={891\ndash 917},
}

\bib{Halupczok}{article}{
      author={Halupczok, K.},
       title={A new bound for the large sieve inequality with power moduli},
        date={2012},
        ISSN={1793-0421},
     journal={Int. J. Number Theory},
      volume={8},
      number={3},
       pages={689\ndash 695},
}

\bib{Huxley10}{article}{
      author={Huxley, M.~N.},
       title={The large sieve inequality for algebraic number fields},
        date={1968},
        ISSN={0025-5793},
     journal={Mathematika},
      volume={15},
       pages={178\ndash 187},
}

\bib{iwakow}{book}{
      author={Iwaniec, H.},
      author={Kowalski, E.},
       title={{A}nalytic {N}umber {T}heory},
      series={American Mathematical Society Colloquium Publications},
   publisher={American Mathematical Society},
     address={Providence},
        date={2004},
      volume={53},
}

\bib{JVL1}{article}{
      author={Linnik, J.~V.},
       title={The large sieve},
        date={1941},
     journal={Doklady Akad. Nauk SSSR},
      volume={36},
       pages={119\ndash 120},
}

\bib{P&S}{book}{
      author={Parshin, A.~N.},
      author={Shafarevich, I.~R.},
       title={Number {T}heory {IV}: {T}ranscendental {N}umbers},
      series={Encyclopaedia of Mathematical Sciences},
   publisher={Springer, Berlin},
        date={1998},
      volume={44},
        ISBN={978-0-8218-4970-5},
}

\bib{ShpZha}{article}{
      author={Shparlinski, I.~E.},
      author={Zhao, L.},
       title={Elliptic curves in isogeny classes},
        date={2018},
     journal={J. Number Theory},
      volume={191},
       pages={194\ndash 212},
}

\bib{Wol}{article}{
      author={Wolke, D.},
       title={On the large sieve with primes},
        date={1971/72},
     journal={Acta Math. Acad. Sci. Hungar.},
      volume={22},
       pages={239\ndash 247},
}

\bib{Zha}{article}{
      author={Zhao, L.},
       title={Large sieve inequality for characters to square moduli},
        date={2004},
     journal={Acta Arith.},
      volume={112},
      number={3},
       pages={297\ndash 308},
}

\end{biblist}
\end{bibdiv}
\bibliographystyle{amsxport}

\vspace*{.5cm}

\noindent\begin{tabular}{p{8cm}p{8cm}}
School of Mathematical Sciences & School of Mathematics and Statistics \\
Beihang University & University of New South Wales \\
Beijing 100191 China & Sydney NSW 2052 Australia \\
Email: {\tt penggao@buaa.edu.cn} & Email: {\tt l.zhao@unsw.edu.au} \\
\end{tabular}

\end{document}